\documentclass[reqno]{amsart}
\usepackage{setspace}
\usepackage{tensor}
\usepackage[compress,square,comma,numbers]{natbib}
\usepackage{hyperref}
\usepackage{times}
\usepackage{amssymb}
\usepackage{tikz}
\usepackage{mathtools}
\usepackage{float}
\usepackage{soul}
\usepackage{xcolor}
\definecolor{edit}{HTML}{CC0000} 
\definecolor{edit2}{HTML}{1850DE}

\makeatletter
\def\bc{\begin{center}}
\def\ec{\end{center}}

\def\s2c{\vskip 2cm}
\def\bt{\begin{Theorem}}
\def\et{\end{Theorem}}
\def\bd{\begin{Definition}}
\def\ed{\end{Definition}}
\def\bl{\begin{Lemma}}
\def\el{\end{Lemma}}
\def\bcor{\begin{Corollary}}
\def\ecor{\end{Corollary}}
\def\bpr{\begin{Proposition}}
\def\epr{\end{Proposition}}

\newtheorem{Lemma}{Lemma}[section]
\newtheorem{Theorem}[Lemma]{Theorem}
\newtheorem{Proposition}[Lemma]{Propostion}
\newtheorem{Definition}[Lemma]{Definition}
\newtheorem{Corollary}[Lemma]{Corollary}
\newtheorem{theorem}{Theorem}[section]
\newtheorem{lemma}[theorem]{Lemma}

\newtheorem{proposition}[theorem]{Proposition}

\newtheorem{definition}[theorem]{Definition}

\newtheorem{remark}[theorem]{Remark}

\def\elsartstyle{%
    \def\normalsize{\@setfontsize\normalsize\@xiipt{14.5}}
    \def\small{\@setfontsize\small\@xipt{13.6}}
    \let\footnotesize=\small
    \def\large{\@setfontsize\large\@xivpt{18}}
    \def\Large{\@setfontsize\Large\@xviipt{22}}
    \skip\@mpfootins = 18\p@ \@plus 2\p@
    \normalsize
} \@ifundefined{square}{}{} \makeatother

\author{ SUBHA PAL and Sarath Sasi$^*$ }

\date{}

\thanks{$^*$Corresponding author's e-mail:sarath@iitpkd.ac.in, \\{\bf Author name and affiliations:}\\
  Sarath Sasi, Department of Mathematics,\\
  Indian Institute of Technology Palakkad,\\
  Kerala, INDIA-678623.\\
 }
\begin{document}

\title{Principal eigenvalues and asymptotic behavior for the weighted $p$-Laplacian with Robin boundary conditions on exterior domains}

\maketitle \vskip .5cm \noindent {\bf Abstract:}
The spectral theory of the $p$-Laplacian is well-developed for classical Dirichlet and Neumann boundary conditions, but the transitional Robin regime on exterior domains remains a largely unexplored territory. This paper provides a comprehensive analysis of the weighted p-Laplacian eigenvalue problem with Robin boundary conditions on the exterior domain $B_1^c \subset \mathbb{R}^N$, with $N>p $, where the weight function belongs to the Lorentz space $F^{N/p}$ and decays at infinity. Under natural assumptions on the weight, we prove the existence, uniqueness, simplicity, and isolation of a positive principal eigenvalue and show the $C_{loc}^{1,\alpha} $ regularity of the associated eigenfunction. We further analyze the dependence of the principal eigenvalue on the Robin parameter  $\beta$ and recover the Neumann and Dirichlet limits as $\beta \to 0^+$ and $\beta \to \infty$, respectively. Our asymptotic analysis reveals a universal far-field decay rate $|x|^{-(N-p)/(p-1)}$ independent of the Robin parameter $\beta$, while the near-boundary structure exhibits explicit $\beta^{1/(p-1)}$ scaling. We then investigate the gradient behavior of the eigenfunction, showing the existence of a unique critical point $r_*$, and provide explicit quantitative bounds on both $r_*$ and the boundary value $\varphi(1)$ in terms of Robin parameter $\beta$. The central contribution of this paper is the derivation of unified gradient estimates that seamlessly connect the near-boundary and asymptotic regions by introducing a characteristic length scale $L \sim \beta^{-1/(N-1)}$. These estimates provide a global description of the eigenfunction’s gradient behavior on the entire exterior domain and quantify how the Robin parameter controls the penetration depth of boundary effects.

\vskip .3cm \noindent {\bf AMS Classification (2020):}{ 35P30, 35J92, 35B40}


\section{Introduction}

The spectral theory of quasilinear elliptic operators on unbounded domains has attracted considerable attention due to its fundamental role in understanding the qualitative behavior of solutions to nonlinear partial differential equations. Of particular importance are eigenvalue problems involving the $p$-Laplacian operator $\Delta_p u := \text{div}(|\nabla u|^{p-2}\nabla u)$, which arises naturally in numerous applications including nonlinear elasticity, fluid dynamics, image processing, population dynamics, and heat transfer with convective boundary conditions. While the spectral properties of the $p$-Laplacian are well-understood for bounded domains and the entire space $\mathbb{R}^N$, the case of exterior domains with Robin boundary conditions remains largely unexplored.
The classical eigenvalue problem for the $p$-Laplacian with weight function $g$ takes the form
$$-\Delta_p u = \lambda g|u|^{p-2}u \quad \text{in } \Omega,$$
supplemented by appropriate boundary conditions. When $\Omega$ is a bounded domain, the existence of a principal eigenvalue and its associated positive eigenfunction is well-established under suitable conditions on the weight function $g$. For the entire space $\mathbb{R}^N$, the situation is more delicate, requiring careful analysis of the competition between the gradient energy and the weight function, particularly when $p<N$.

Exterior domains, defined as complements of bounded regions, present unique challenges that combine aspects of both bounded and unbounded domain theories. Unlike bounded domains where compactness arguments readily apply, or the entire space where translation invariance provides powerful tools, exterior domains require sophisticated techniques to handle both local boundary phenomena and global decay properties simultaneously. There have been many studies on the weighted eigenvalue problem for the Dirichlet $p$-Laplace operator on exterior domains. In \cite{ads2015} the existence of a first principal eigenvalue is investigated in the setting of Beppo-Levi space - the completion of  $C_c^\infty(B_1^c)$ with respect to the norm $\|\nabla u\|_p=\displaystyle{\int_{B_1^c}}|\nabla u|^p \rm{d}x$.
A further generalization of these results can be found in \cite{dhs2018}.
Subsequently, Chhetri and Dr\'{a}bek \cite{cd14}, provided detailed estimates about the behaviour of the eigenfunctions both near the boundary and at infinity, establishing precise decay rates that depend on the spatial dimension and the nonlinearity parameter.

Recently,  Anoop and Biswas \cite{anoop19} studied Neumann eigenvalue problems  on exterior domains, establishing existence and properties of principal eigenvalues for weight functions in certain Lorentz and weighted Lebesgue spaces.

However, the intermediate case of Robin boundary conditions, which naturally interpolates between Neumann ($\beta = 0$) and Dirichlet ($\beta = \infty$) conditions through the parameter $\beta > 0$, has remained unaddressed. Robin boundary conditions frequently arise in physical applications involving heat or mass transfer across boundaries, where the flux is proportional to the difference between the solution and some external reference value. Understanding how the Robin parameter $\beta$ influences the spectral properties provides crucial insight into the transition between different physical regimes.

In this paper, we provide a comprehensive analysis of the Robin eigenvalue problem
\begin{equation}\label{eq:main-intro}
\begin{cases}
-\Delta_p u = \lambda g|u|^{p-2}u & \text{in } B_1^c,\\
|\nabla u|^{p-2}\frac{\partial u}{\partial \nu} + \beta|u|^{p-2}u = 0 & \text{on } \partial B_1,
\end{cases}
\end{equation}
where $B_1^c = \mathbb{R}^N \setminus \overline{B_1}$ is the exterior of the closed unit ball, $1 < p < N$, $\beta > 0$ is the Robin parameter, $\nu$ is the outward unit normal to $\partial B_1$, and $g$ is a weight function belonging to the space $F^{N/p}$ (the closure of $C_c^\infty(B_1^c)$ in the Lorentz space $L^{N/p,\infty}(B_1^c)$).

Our first main contribution establishes the fundamental spectral properties of problem \eqref{eq:main-intro}, extending the variational framework to the Robin setting:

\begin{theorem}\label{thm:principal-intro}
Let $1<p<N$, and $g \in F^{N/p}$ with $g^+ \not\equiv 0$. Then for the eigenvalue problem \eqref{eq:main-intro}:

\begin{itemize}
    \item[(a)] \textbf{(Existence and Uniqueness)} The value
    \[
    \lambda_1 = \inf\left\{\int_{B_1^c} |\nabla \phi|^p + \beta\int_{\partial B_1} |\phi|^p : \phi \in W^{1,p}(B_1^c), \int_{B_1^c} g|\phi|^p = 1\right\}
    \]
    is achieved and defines the unique positive principal eigenvalue.
    
    \item[(b)] \textbf{(Boundedness)} Every eigenfunction $\phi$ satisfies $\|\phi\|_{L^\infty(B_1^c)} \leq C\|\phi\|_{W^{1,p}(B_1^c)}$.
    
    \item[(c)] \textbf{(Higher Regularity)} If additionally
$g \in L^q_{loc}(B_1^c)$ for some $q > \frac{Np}{p-1}$, then every eigenfunction belongs to $C^{1,\alpha}_{loc}(\overline{B_1^c})$ 
for some $\alpha \in (0,1)$.
    
    \item[(d)] \textbf{(Simplicity)} The principal eigenvalue $\lambda_1$ is simple.
    
    \item[(e)] \textbf{(Isolation)} The principal eigenvalue $\lambda_1$ is isolated in the spectrum.
\end{itemize}
\end{theorem}

The proof of Theorem \ref{thm:principal-intro} requires overcoming significant technical challenges related to the lack of compactness in the exterior domain setting. For the existence and isolation parts, we adapt the variational framework from Anoop and Biswas \cite{anoop19}, utilizing their compactness result (Proposition 4.2) which establishes that the weight functional $G(u) = \int_{B_1^c} g|u|^p$ is compact on $W^{1,p}(B_1^c)$ for weights in $F^{N/p}$. We also follow their approach for verifying the Palais-Smale condition, with appropriate modifications to handle the Robin boundary term $\beta\int_{\partial B_1}|u|^p$ in the energy functional. The key novelty lies in the uniqueness argument, which requires careful application of Picone's identity to the Robin setting, and the simplicity proof, where the Robin boundary conditions necessitate a detailed analysis to verify that distinct eigenfunctions must be scalar multiples. The regularity analysis combines interior regularity techniques for quasilinear elliptic equations with boundary regularity theory for Robin conditions, employing Moser iteration adapted to the exterior domain setting for boundedness and Lieberman's theory for $C^{1,\alpha}$ estimates near the Robin boundary.

\vspace{.3cm}

\noindent\textbf{Parameter dependence and limiting behavior.} Having established the existence and fundamental properties of the principal eigenvalue $\lambda_1$ for each fixed Robin parameter $\beta > 0$, a natural question arises: how do $\lambda_1$ and the corresponding eigenfunction $\phi_1$ vary as $\beta$ changes? This question is particularly important for understanding the transition between classical boundary conditions, as the Robin problem interpolates between Neumann conditions ($\beta \to 0^+$) and Dirichlet conditions ($\beta \to \infty$).

Following the framework developed by Drábek and Rasouli \cite{dr10} for domains 
of finite measure and adapting it to the exterior domain setting, we conclude that the mapping $\beta \mapsto \lambda_1(\beta)$ is concave, strictly increasing, and continuously differentiable on $(0,\infty)$, with explicit derivative formula
$$\frac{d\lambda_1}{d\beta} = \int_{\partial B_1} \phi_1(\beta)^p \, dH^{N-1}.$$
Moreover, the eigenfunction mapping $\beta \mapsto \phi_1(\beta)$ is continuous from $(0,\infty)$ into $W^{1,p}(B_1^c)$. The limiting behavior exhibits a fundamental distinction from the bounded domain case: as $\beta \to 0^+$, we recover the Neumann eigenvalue $\lambda_1^N > 0$ of Anoop-Biswas \cite{anoop19}, which is strictly positive in exterior domains. As $\beta \to \infty$, we recover the Dirichlet eigenvalue $\lambda_1^D$ of Chhetri-Drábek \cite{cd14}, with strict inequality $\lambda_1(\beta) < \lambda_1^D$ for all finite $\beta$. The complete analysis of these parameter-dependence properties and the limiting behavior, is presented in Remark \ref{rem:beta_dependence} following Theorem \ref{thm:principal-intro}.

With the spectral properties and parameter-dependence structure established, we turn to the asymptotic and near boundary behavior of eigenfunctions.
The next result provides precise asymptotic behavior of eigenfunctions, revealing how the Robin parameter $\beta$ influences both the behaviour of solutions near the boundary  and far-field decay:

\begin{theorem}[Asymptotic Behavior]\label{thm:asymptotic-intro}
Let $1<p<N$ and $g \in F^{N/p}$ with $g^+ \not\equiv 0$. Let $\phi$ be the positive principal eigenfunction corresponding to eigenvalue $\lambda_1$ of problem \eqref{eq:main-intro}. Then:

\begin{itemize}
    \item[(a)] \textbf{(General Decay at Infinity)} Suppose $r_0 > 1$, $l>p$ and there exists a constant $C_0>0$ such that $0 < g(x) \leq C_0|x|^{-l}$ for $|x| \geq r_0$. Then there exist constants $0 < C_1 < C_2$ such that for all $x \in B_{r_0}^c$:
    \begin{align}
    \frac{C_1}{|x|^{(N-p)/(p-1)}} \leq \phi(x) \leq \frac{C_2}{|x|^{(N-p)/(p-1)}} \label{eq:general_decay-intro}
    \end{align}
    
    \item[(b)] \textbf{(Radial Symmetry)} If additionally $g(x)$ is radially symmetric, then the principal eigenfunction $\phi$ is radial, i.e. there exists $\varphi:[1,\infty)\to \mathbb{R}$ such that $\phi(x) = \varphi(|x|)$ for all $x \in B_1^c$.
        
    \item[(c)] \textbf{(Boundary Estimates)} Under the assumption of part (b), there exist constants $K > 0$ and $\alpha \in (0,1)$ such that for all $x \in B_1^c$ with $t := |x| - 1$ sufficiently small:
    \begin{align}
    \varphi(1)(1 + \beta^{1/(p-1)} t) - K t^{1+\alpha} &\leq \phi(x) \leq \varphi(1)(1 + \beta^{1/(p-1)} t) + K t^{1+\alpha} \label{eq:boundary_bounds}
    \end{align}
\end{itemize}
\end{theorem}

\noindent Theorem \ref{thm:asymptotic-intro} reveals several expected features. The far-field decay rate \eqref{eq:general_decay-intro} is independent of the Robin parameter $\beta$, confirming that boundary conditions do not influence asymptotic behavior at infinity. In contrast, the near-boundary behavior \eqref{eq:boundary_bounds} explicitly depends on $\beta^{1/(p-1)}$, showing how the Robin parameter controls the boundary layer structure. The transition between these regimes suggests the existence of a characteristic length scale that governs the spatial extent of boundary effects.

To make this transition precise, we first establish gradient estimates that quantify how $|\nabla\phi|$ behaves near the boundary and at infinity. These estimates not only reveal the existence of a unique critical point $r_* > 1$ where the eigenfunction attains its maximum, but also provide the technical foundation for the complete characterization.

\begin{theorem}[Gradient Estimates]\label{thm:gradient-near-intro}
Let $1<p<N$ and let $g \in F^{N/p}\cap L^{N/p}(B_1^c)$ be radially symmetric with $g(x) > 0$ for all $x \in B_1^c$. Further assume that there exist constants $l>p$ and $C_0>0$ such that $0 < g(x) \leq C_0|x|^{-l}$ for $|x| \geq 1$. Let $\phi(x)=\varphi(|x|)$ be the positive principal eigenfunction corresponding to eigenvalue $\lambda_1$ of problem \eqref{eq:main-intro}. Then:

\begin{itemize}
    \item[(a)] \textbf{(Gradient estimate near boundary)} There exist constants $0 < m < M$, $C > 0$, and $\delta > 0$ such that for all $x \in B_1^c$ with $\mathrm{dist}(x,\partial B_1) < \delta$:
    \begin{align*}
    m\beta^{1/(p-1)}\varphi(1) &\leq |\nabla \phi(x)| \leq M\beta^{1/(p-1)}\varphi(1)\left(1 + C\cdot\mathrm{dist}(x,\partial B_1)^{\alpha}\right)
    \end{align*}
    where $\alpha \in (0,1)$ is the Hölder exponent from Theorem \ref{thm:principal-intro}.

    \item[(b)] \textbf{(Gradient estimate near infinity)} There exists a unique $r_*>1$ such that $\varphi'(r_*) = 0$. Moreover, there exist constants $0 < \hat{C}_1 < \hat{C}_2$ such that for all $x \in B_{2r_*}^c$:
    \begin{align*}
    \hat{C}_1|x|^{-(N-1)/(p-1)} &\leq |\nabla \phi(x)| \leq \hat{C}_2(\log |x|)^{(N-p)/(N(p-1))}|x|^{-(N-1)/(p-1)}
    \end{align*}
\end{itemize}
\end{theorem}

\noindent Theorem \ref{thm:gradient-near-intro} establishes the local gradient structure and identifies the critical point $r_*$ as the unique location where the eigenfunction achieves its maximum. However, the dependence of both $r_*$ and the boundary value $\varphi(1)$ on the Robin parameter $\beta$ remains implicit. Our next result provides explicit quantitative bounds that reveal how $\beta$ controls both the magnitude of the eigenfunction at the boundary and the spatial location of its maximum.

\begin{theorem}[Boundary Value and Critical Point Characterization]\label{thm:critical-intro}
Let $1<p<N$ and let $g \in F^{N/p}\cap L^{N/p}(B_1^c)$ be radially symmetric with $g(x) > 0$ for all $x \in B_1^c$. Let $\phi(x) = \varphi(|x|)$ be the positive principal eigenfunction corresponding to eigenvalue $\lambda_1$ of problem \eqref{eq:main-intro}, and let $r_* > 1$ denote the unique critical radius where $\varphi'(r_*) = 0$ established in Theorem \ref{thm:gradient-near-intro}(b). Then:

\begin{itemize}
    \item[(a)] \textbf{(Bounds on Boundary Value)} Assume there exist constants $C_0 > 0$ and $l > N$ such that $0 < g(x) \leq C_0|x|^{-l}$ for $|x| \geq 1$. Then:
    \begin{equation}\label{eq:phi1_bounds-intro}
    \frac{C_g}{1 + \frac{p-1}{N-p}\beta^{1/(p-1)}} \leq \varphi(1) < \left(\frac{\lambda_1}{\beta\omega_{N-1}}\right)^{1/(p-1)}
    \end{equation}
    where $C_g = \left(\omega_{N-1}\int_1^\infty r^{N-1}g(r)\,dr\right)^{-1/p}$ and $\omega_{N-1}$ is the surface area of the unit sphere in $\mathbb{R}^N$.
    
    \item[(b)] \textbf{(Bounds on Critical Point)} The critical radius $r_*$ satisfies:
    \begin{equation}\label{eq:rstar_bounds-intro}
    \left(1 + \frac{N\beta}{\lambda_1 g_{\max}\left[1 + \frac{p-1}{N-p}\beta^{1/(p-1)}\right]^{p-1}}\right)^{1/N} \leq r_* \leq \left(1 + \frac{N\beta}{\lambda_1 g_{\min}}\right)^{1/N},
    \end{equation}
    where $g_{\min} = \min_{s \in [1,r_*]} g(s)$ and $g_{\max} = \max_{s \in [1,r_*]} g(s)$.
\end{itemize}
\end{theorem}

\noindent Theorem \ref{thm:critical-intro} provides explicit bounds on the boundary value and critical point location, revealing how the Robin parameter $\beta$ influences the eigenfunction's magnitude and geometry. However, understanding the complete spatial structure of $|\nabla\phi|$ throughout the exterior domain requires a unified framework that seamlessly interpolates between the near-boundary regime (where gradients scale as $\beta^{1/(p-1)}$), the intermediate transition region, and the far-field asymptotic regime (where gradients decay as $|x|^{-(N-1)/(p-1)}$).

Our final and most technical result addresses this challenge through the introduction of a characteristic length scale $L \sim \beta^{-1/(N-1)}$ and carefully constructed transition functions. This framework reveals how boundary effects penetrate into the domain with a well-defined spatial extent controlled by $\beta$, and how the critical point $r_*$ creates a natural partition of the domain into growth and decay phases:

\begin{theorem}[Unified Gradient Estimates]\label{thm:unified-intro}
Let $1<p<N$ and let $g \in F^{N/p}\cap L^{N/p}(B_1^c)$ be radially symmetric with $g(x) > 0$ for all $x \in B_1^c$. Further assume that there exist constants $l>p$ and $C_0>0$ such that $0 < g(x) \leq C_0|x|^{-l}$ for $|x| \geq 1$. Let $\phi(x)=\varphi(|x|)$ be the positive principal eigenfunction corresponding to eigenvalue $\lambda_1$ of problem \eqref{eq:main-intro}.

Let $0 < \beta_0 < \beta_1 < \infty$ and $\delta>0$ be sufficiently small. For $\beta\in[\beta_0, \beta_1]$, define $L = \beta^{-1/(N-1)}$ and $\sigma(r) := \left(\frac{|r - r_*|}{|r - r_*| + \delta}\right)^{1/(p-1)}$. Then there exist constants $C_1(\beta), C_2(\beta) > 0$ such that for all $x \in B_1^c$:
\begin{align*}
|\nabla\phi(x)| &\geq C_1 \cdot \sigma(|x|) \cdot \Big[\tau(|x|-1) \beta^{1/(p-1)} + (1 - \tau(|x|-1)) |x|^{-(N-1)/(p-1)}\Big] \\
|\nabla\phi(x)| &\leq C_2\Big[\tau(|x|-1) \beta^{1/(p-1)} + (1 - \tau(|x|-1)) |x|^{-(N-1)/(p-1)}h(|x|)\Big] 
\end{align*}
where $\tau(r) = (1 + (r/L)^\gamma)^{-1}$ with $\gamma \geq 2$, and $h(r) = \max\{1, (\log r)^{(N-p)/(N(p-1))}\}$. The characteristic length scale $L$ quantifies the penetration depth of boundary effects, and the function $\sigma$ ensures the lower bound vanishes at the critical point $r_*$.
\end{theorem}

\noindent The unified gradient estimates in Theorem \ref{thm:unified-intro} represent our most significant technical innovation, providing a complete description of how Robin boundary conditions influence eigenfunction gradients throughout the entire exterior domain. The characteristic length scale $L \sim \beta^{-1/(N-1)}$ emerges naturally from the analysis and has clear physical interpretation: stronger Robin coupling (larger $\beta$) confines boundary effects to a thinner layer near $\partial B_1$, while weaker coupling allows boundary-induced variations to penetrate further into the domain.


The remainder of this paper is organized as follows. Section \ref{pre} establishes the functional framework and proves key preliminary results including Lorentz space embeddings, compactness properties, and the strong maximum principle for Robin problems on exterior domains. Section \ref{main} contains the complete proofs of Theorems \ref{thm:principal-intro}--\ref{thm:unified-intro}, with particular attention to the technical challenges arising from the exterior domain geometry and Robin boundary conditions. The proofs include detailed parameter-dependence analysis, geometric characterization of the critical point, and the development of unified gradient estimates using novel transition functions and the characteristic length scale framework.

\section{Preliminaries}\label{pre}
Let $B_1$ denote the unit closed ball in $\mathbb{R}^N$ and $B_1^c := \mathbb{R}^N \setminus B_1$ be the exterior domain. For $1 < p < \infty$, we denote by $W^{1,p}(B_1^c)$ the usual Sobolev space equipped with the norm
\[
\|u\|_{W^{1,p}} = \left(\int_{B_1^c} |u|^p + |\nabla u|^p\right)^{1/p}.
\]
For a measurable function $f$ defined on $B_1^c$ and $s > 0$, let $E_f(s) = \{x \in B_1^c : |f(x)| > s\}$. The distribution function $\alpha_f$ of $f$ is defined as $\alpha_f(s) = |E_f(s)|$ for $s > 0$, where $|\cdot|$ denotes the Lebesgue measure. The one-dimensional decreasing rearrangement $f^*$ of $f$ is defined as
\[
f^*(t) = \inf\{s > 0 : \alpha_f(s) < t\}, \quad \text{for } t > 0.
\]
For $(p,q) \in [1,\infty) \times [1,\infty]$, we define the Lorentz space $L^{p,q}(B_1^c)$ as
\[
L^{p,q}(B_1^c) := \left\{f \text{ measurable} : |f|_{(p,q)} < \infty\right\},
\]
where
$$
|f|_{(p,q)} := \begin{cases}
\left(\int_0^\infty \left[t^{1/p-1/q}f^*(t)\right]^q dt\right)^{1/q}, & 1 \leq q < \infty,\\
\sup_{t>0} t^{1/p}f^*(t), & q = \infty.
\end{cases}
$$
For $N > p$, we consider the closed subspace of $L^{N/p,\infty}(B_1^c)$ defined by
\[
F^{N/p} := \overline{C_c^\infty(B_1^c)}^{L^{N/p,\infty}(B_1^c)}.
\]

\begin{definition}
We say a function $g \in L^1_{loc}(B_1^c)$ belongs to the class $\mathcal{A}$, if $\text{supp}(g^+)$ has a positive measure and $g \in F^{N/p}$ with $N > p$.
\end{definition}
\noindent For $p \in (1,\infty)$, we consider the eigenvalue problem (\ref{eq:main-intro}). We say $\lambda \in \mathbb{R}$ is an eigenvalue of \eqref{eq:main-intro} if there exists $u \in W^{1,p}(B_1^c) \setminus \{0\}$ satisfying
\[
\int_{B_1^c} |\nabla u|^{p-2}\nabla u \cdot \nabla v + \beta\int_{\partial B_1} |u|^{p-2}uv = \lambda \int_{B_1^c} g|u|^{p-2}uv,
\]
for all $v \in W^{1,p}(B_1^c)$. In this case, $u$ is called an eigenfunction corresponding to $\lambda$.

\begin{lemma}[Generalized Hardy-Sobolev inequality]\label{lem:hardy-sobolev}
Let $N > p$. If $g \in L^{N/p,\infty}(B_1^c)$, then there exists a constant $C = C(N,p)$ such that
\begin{equation}
\label{Gen_Holder}
    \left|\int_{B_1^c} g|u|^p \, dx\right| \leq C\|g\|_{(N/p,\infty)} \int_{B_1^c} |\nabla u|^p \, dx,
\end{equation}
for all $u \in W^{1,p}(B_1^c)$. In particular, this inequality holds for $g \in F^{N/p}$.
\end{lemma}

\begin{proof}
See Remark 3.3 in \cite{anoop19}. The second statement follows since $F^{N/p} \subset L^{N/p,\infty}(B_1^c)$.
\end{proof}

\begin{lemma}\label{lem:zeroset}
Let $g \in F^{N/p}$ with $N > p$ and consider a solution $(\lambda,\phi)$ of \eqref{eq:main-intro} with $\phi \geq 0$ in $B_1^c$. Then the set $Z := \{x \in B_1^c : \phi(x) = 0\}$ has zero $W^{1,p}$-capacity. Consequently, $H^s(Z) = 0$ for all $s > N-p$, where $H^s$ denotes the $s$-dimensional Hausdorff measure.
\end{lemma}

\begin{proof}
The proof is adapted from \cite[Theorem 2.4]{lp06} which itself is inspired by \cite[Theorem 10.9]{GT2001}. Let $\xi \in C_c^\infty(B_1^c)$ and $\delta > 0$. We have
\begin{align*}
    \int_{B_1^c} & \left|\nabla \log\left(1 + \frac{\phi}{\delta}\right)\right|^p\xi^p = \int_{B_1^c} \left|\frac{\nabla \phi}{\phi + \delta}\right|^p\xi^p \\
    & = \int_{B_1^c} |\nabla \phi|^{p-2}\nabla \phi \cdot \frac{\nabla \phi}{(\phi + \delta)^p}\xi^p = -\frac{1}{p-1}\int_{B_1^c} |\nabla \phi|^{p-2}\nabla \phi \cdot \nabla\left(\frac{\xi^p}{(\phi + \delta)^{p-1}}\right) \\
    & = \frac{1}{p-1}\int_{B_1^c} \lambda g|\phi|^{p-2}\phi \frac{\xi^p}{(\phi + \delta)^{p-1}} + \frac{1}{p-1}\int_{\partial B_1} \beta|\phi|^{p-2}\phi \frac{\xi^p}{(\phi + \delta)^{p-1}}.
\end{align*}
Using the inequality $\frac{|\phi|^{p-2}\phi}{(\phi + \delta)^{p-1}} \leq 1$ and trace theorem, we get

$$(p-1)\int_{B_1^c} \left|\nabla \log\left(1 + \frac{\phi}{\delta}\right)\right|^p\xi^p \leq \lambda\int_{B_1^c} g(x)\xi^p + C\int_{B_1^c} |\nabla\xi^p|.$$
Since $g \in F^{N/p}$ and $\xi$ is a test function with compact support, the first term on the right-hand side is bounded  by Lemma \ref{lem:hardy-sobolev}. The second term is bounded since $\xi \in C_c^\infty(B_1^c)$. Therefore, there exists a constant $C$ independent of $\delta$ such that
\[
\int_{B_1^c} \left|\nabla \log\left(1 + \frac{\phi}{\delta}\right)\right|^p\xi^p \leq C
\]
for all $\delta > 0$. Assume by contradiction that $Z$ has positive $W^{1,p}$-capacity. Consider an arbitrary open set $\omega \subset B_1^c$ intersecting $Z$. By Poincaré inequality for sets of positive capacity (see  \cite[Corollary 4.5.2]{Ziemer1989}) we have
\[
\int_\omega \left|\log\left(1 + \frac{\phi}{\delta}\right)\right|^p \leq C.
\]
Since both $\delta$ and $\omega$ are arbitrary, and we have uniform bounds independent of $\delta$, this would force $\phi \equiv 0$ in $B_1^c$, contradicting that $\phi$ is a non-zero solution. Therefore $Z$ must have zero $W^{1,p}$-capacity. Now, we can conclude that $H^s(Z) = 0$ for all $s > N-p$.
\end{proof}

\begin{proposition}[Strong Maximum Principle]\label{prop:smp}
Let $\phi$ be a non-negative function in $W^{1,p}(B_1^c)$ satisfying
\begin{equation}
\int_{B_1^c} |\nabla \phi|^{p-2}\nabla \phi \cdot \nabla v + \beta\int_{\partial B_1} |\phi|^{p-2}\phi v = \lambda\int_{B_1^c} g|\phi|^{p-2}\phi v
\end{equation}
for all $v \in W^{1,p}(B_1^c)$, where $\beta > 0$ and $g \in F^{N/p}$ with $N > p$. Then either $\phi \equiv 0$ or $\phi > 0$ a.e. in $B_1^c$.
\end{proposition}

\begin{proof}
By Lemma \ref{lem:zeroset}, if $\phi \not\equiv 0$, then its set of zeros $Z := \{x \in B_1^c : \phi(x) = 0\}$ has zero $W^{1,p}$-capacity. Since sets of zero capacity cannot disconnect a domain, and $\phi$ is continuous in the complement of a set of zero capacity, we conclude that $\phi > 0$ a.e. in $B_1^c$.
\end{proof}

\noindent We define the functionals $J,G: W^{1,p}(B_1^c) \to \mathbb{R}$ as
\[
J(u) = \int_{B_1^c} |\nabla u|^p + \beta\int_{\partial B_1} |u|^p \quad \text{and} \quad G(u) = \int_{B_1^c} g|u|^p.
\]
One can easily verify that $J,G \in C^1(W^{1,p}(B_1^c); \mathbb{R})$ and for $u,v \in W^{1,p}(B_1^c)$,
\[
\langle J'(u),v \rangle = p\int_{B_1^c} |\nabla u|^{p-2}\nabla u \cdot \nabla v + p\beta\int_{\partial B_1} |u|^{p-2}uv,
\]
\[
\langle G'(u),v \rangle = p\int_{B_1^c} g|u|^{p-2}uv,
\]
where $\langle \cdot,\cdot \rangle$ denotes the duality action.
Let $N_g := \{u \in W^{1,p}(B_1^c) : G(u) = 1\}$ denote the constraint manifold.
We recall the following proposition from \cite[Proposition 4.3]{anoop19}.

\begin{proposition}\label{ref:compactness}
If $g \in \mathcal{A}$, then $G$ and $G'$ are compact on $W^{1,p}(B_1^c)$.
\end{proposition}

\noindent Next we state and prove an estimate which plays a cruicial role in the proof of the main existence result.

\begin{lemma}\label{lem:poincare}
Let $g \in \mathcal{A}$. Then there exists $m > 0$ such that
\[
\int_{B_1^c} |\nabla u|^p + \beta\int_{\partial B_1} |u|^p \geq m\int_{B_1^c} |u|^p, \quad \forall u \in N_g.
\]
\end{lemma}

\begin{proof}
We prove this by contradiction. Suppose the inequality is false. Then for each $n \in \mathbb{N}$, there exists $\phi_n \in N_g$ such that
\[
\int_{B_1^c} |\nabla \phi_n|^p + \beta\int_{\partial B_1} |\phi_n|^p \leq \frac{1}{n}\int_{B_1^c} |\phi_n|^p.
\]
Set $k_n = \int_{B_1^c} |\phi_n|^p$. Then from the above inequality:
\[
\int_{B_1^c} |\nabla \phi_n|^p \leq \frac{k_n}{n} \quad \text{and} \quad \beta\int_{\partial B_1} |\phi_n|^p \leq \frac{k_n}{n}.
\]
Since $\phi_n \in N_g$, we have $\int_{B_1^c} g|\phi_n|^p = 1$. Using the generalized Hölder's inequality \eqref{Gen_Holder} we have:
\[
1 = \left|\int_{B_1^c} g|\phi_n|^p\right| \leq \|g\|_{(N/p,\infty)} \|\phi_n\|_p^p \leq C_1k_n.
\]
Therefore, $k_n \geq C_2 > 0$ for all $n \in \mathbb{N}$. 
Now set $\psi_n = \phi_n/k_n^{1/p}$. Then,
\begin{align*}
    \int_{B_1^c} |\psi_n|^p = 1, &\quad \int_{B_1^c} g|\psi_n|^p = \frac{1}{k_n} \leq C_3 \text{ for all } n \in \mathbb{N}, \\
   \int_{B_1^c} |\nabla \psi_n|^p &= \frac{1}{k_n}\int_{B_1^c} |\nabla \phi_n|^p \leq \frac{1}{nk_n} \to 0 \text{ as } n \to \infty, \\
   \beta\int_{\partial B_1} |\psi_n|^p &= \frac{\beta}{k_n}\int_{\partial B_1} |\phi_n|^p \leq \frac{1}{nk_n} \to 0 \text{ as } n \to \infty.
\end{align*}
Due to these properties, $\{\psi_n\}$ is bounded in $W^{1,p}(B_1^c)$. Therefore, by the reflexivity of $W^{1,p}(B_1^c)$, there exists $\psi \in W^{1,p}(B_1^c)$ such that $\psi_n \rightharpoonup \psi$ weakly in $W^{1,p}(B_1^c)$.
From the weak lower semicontinuity of norms, we have:
\[
\int_{B_1^c} |\nabla \psi|^p \leq \liminf_{n \to \infty} \int_{B_1^c} |\nabla \psi_n|^p = 0
\]
and 
\[
\beta\int_{\partial B_1} |\psi|^p \leq \liminf_{n \to \infty} \beta\int_{\partial B_1} |\psi_n|^p = 0.
\]
Therefore, $\psi$ is constant in $B_1^c$, and since $\beta > 0$, we must have $\psi = 0$ on $\partial B_1$. This implies $\psi \equiv 0$ in $B_1^c$.

On the other hand, by Proposition \ref{ref:compactness}, the functional $G$ is compact on $W^{1,p}(B_1^c)$. Hence, $\psi_n \rightharpoonup \psi$ weakly implies $G(\psi_n) \to G(\psi)$. Since $G(\psi_n) = \int_{B_1^c} g|\psi_n|^p = \frac{1}{k_n} > 0$, we must have $G(\psi) > 0$, which implies $\psi \not\equiv 0$.
This contradiction proves the lemma.
\end{proof}

\begin{remark}\label{re1}
The constraint manifold $N_g$ admits a natural differential structure. Since $G'(u)\neq 0$ for all $u\in N_g$ (indeed, $\langle G'(u),u\rangle=p>0$), the level set $N_g=G^{-1}(1)$ is a $C^1$ submanifold of $W^{1,p}(B_1^c)$ by the implicit function theorem. At each $u\in N_g$, the tangent space is
\[
T_uN_g=\{v\in W^{1,p}(B_1^c):\langle G'(u),v\rangle=0\}.
\]
The norm of the constrained differential of $J$ at $u\in N_g$ is given by
\begin{equation}
\|dJ(u)\|_{T_u^*N_g}=\inf_{\lambda\in\mathbb{R}}\|J'(u)-\lambda G'(u)\|, \label{constrained_gradient}
\end{equation}
with the minimizing $\lambda$ corresponding to the Lagrange multiplier in the Euler–Lagrange equation. In particular, $u$ is a critical point of $J$ on $N_g$ if and only if there exists $\lambda\in\mathbb{R}$ such that $J'(u)=\lambda G'(u)$, which is precisely the weak formulation of the Robin eigenvalue problem \eqref{eq:main-intro}.
\end{remark}

\begin{definition}[Palais–Smale Condition on a Manifold]\label{def:ps_manifold}
Let $\mathcal{M}$ be a $C^1$ submanifold of a Banach space $X$, and let $\Phi \in C^1(\mathcal{M},\mathbb{R})$.  
We say that $\Phi$ satisfies the Palais–Smale condition at level $c\in\mathbb{R}$ if every sequence $(u_n)\subset \mathcal{M}$ such that
\[
\Phi(u_n)\to c \quad \text{and} \quad \|d\Phi(u_n)\|_{T_{u_n}^*\mathcal{M}} \to 0
\]
has a convergent subsequence in $\mathcal{M}$.  
If this holds for all $c\in\mathbb{R}$, then $\Phi$ is said to satisfy the (PS) condition on $\mathcal{M}$.
\end{definition}

\begin{lemma}\label{lem:ps}
Let $g \in F^{N/p}$ with $N > p$. Then the functional $J$ satisfies the Palais–Smale condition on the constraint manifold $N_g$.
\end{lemma}

\begin{proof}
Let $(u_n) \subset N_g$ be a Palais–Smale sequence, i.e.
\[
J(u_n) \to \lambda \in \mathbb{R}, \qquad \|dJ(u_n)\| \to 0 \quad \text{as } n \to \infty.
\]
By the characterization of the constrained gradient \eqref{constrained_gradient}, there exists a sequence $(\lambda_n) \subset \mathbb{R}$ such that
\[
A_{\lambda_n}(u_n) := J'(u_n) - \lambda_n G'(u_n) \to 0 \quad \text{in } (W^{1,p}(B_1^c))^*.
\]
Using Lemma~\ref{lem:poincare} and the reflexivity of $W^{1,p}(B_1^c)$, we may assume (after passing to a subsequence) that $u_n \rightharpoonup u$ weakly in $W^{1,p}(B_1^c)$. Moreover,
\[
\langle J'(u_n) - \lambda_n G'(u_n), u_n \rangle = p(J(u_n) - \lambda_n) \to 0,
\]
so that $\lambda_n \to \lambda$.
Next, we examine the difference
\[
\langle J'(u_n) - J'(u), u_n - u \rangle
= \langle A_{\lambda_n}(u_n) - A_\lambda(u), u_n - u \rangle
+ \langle \lambda_n G'(u_n) - \lambda G'(u), u_n - u \rangle.
\]
The first term tends to zero because $A_{\lambda_n}(u_n)\to 0$ and $A_\lambda(u)$ is bounded; the second term vanishes due to weak convergence of $u_n$ and compactness of $G'$. Hence,
\[
\langle J'(u_n) - J'(u), u_n - u \rangle \to 0.
\]
Expanding this quantity gives
\begin{align*}
\frac{1}{p}\langle J'(u_n) - J'(u), u_n - u \rangle
&= \|\nabla u_n\|_p^p + \beta \|u_n\|_{L^p(\partial B_1)}^p 
   + \|\nabla u\|_p^p + \beta \|u\|_{L^p(\partial B_1)}^p \\
&\quad - \int_{B_1^c} \big(|\nabla u_n|^{p-2}\nabla u_n \cdot \nabla u
   + |\nabla u|^{p-2}\nabla u \cdot \nabla u_n\big) \\
&\quad - \beta \int_{\partial B_1}\big(|u_n|^{p-2}u_n u + |u|^{p-2}u u_n\big).
\end{align*}
Applying Hölder’s inequality to the mixed terms yields
\[
\frac{1}{p}\langle J'(u_n) - J'(u), u_n - u \rangle
\geq (\|\nabla u_n\|_p^{p-1} - \|\nabla u\|_p^{p-1})(\|\nabla u_n\|_p - \|\nabla u\|_p)
\]
\[
\quad + \beta \big(\|u_n\|_{L^p(\partial B_1)}^{p-1} - \|u\|_{L^p(\partial B_1)}^{p-1}\big)\big(\|u_n\|_{L^p(\partial B_1)} - \|u\|_{L^p(\partial B_1)}\big).
\]
Since the left-hand side tends to zero, it follows that
\[
\|\nabla u_n\|_p \to \|\nabla u\|_p, \qquad 
\|u_n\|_{L^p(\partial B_1)} \to \|u\|_{L^p(\partial B_1)}.
\]
Thus, the weak convergence $u_n \rightharpoonup u$ in $W^{1,p}(B_1^c)$, together with the uniform convexity of $L^p$ spaces, implies
\[
\nabla u_n \to \nabla u \quad \text{in } L^p(B_1^c)^N.
\]
Finally, Lemma~\ref{lem:poincare} gives strong convergence $u_n \to u$ in $W^{1,p}(B_1^c)$. Hence $J$ satisfies the Palais–Smale condition on $N_g$.
\end{proof}


\section{Proofs of Main Results}\label{main}

\subsection{Proof of Theorem \ref{thm:principal-intro}}

The proof follows the variational framework established for Neumann problems in \cite{anoop19}, with necessary modifications for the Robin boundary conditions. We outline the key steps and provide detailed arguments only where the Robin setting requires new techniques.

\vspace{.4cm}

    \noindent \textbf{(a)} \textit{The existence:} The existence of $\lambda_1$ follows by standard variational arguments analogous to those in Anoop and Biswas \cite{anoop19}. The key steps are: (i) establishing that the minimizing sequence $(\phi_n)$ for $J$ on $N_g$ is bounded in $W^{1,p}(B_1^c)$ using Lemma \ref{lem:poincare}, (ii) extracting a weakly convergent subsequence by reflexivity, (iii) showing $N_g$ is weakly closed via the compactness of $G$ (Proposition \ref{ref:compactness}), and (iv) applying weak lower semicontinuity of $J$. The positivity $\lambda_1 > 0$ follows since $J(\phi) > 0$ for all $\phi \in N_g$.

\vspace{.4cm}

\noindent \textit{The Principality:} The argument that any eigenfunction corresponding to $\lambda_1$ is positive follows the same pattern as in \cite{anoop19}. The key observation is that $|\Phi|$ satisfies the same eigenvalue equation with a supersolution-type inequality. Since $g \in F^{N/p}$ ensures $g|\Phi|^{p-1} \in L^1_{loc}(B_1^c)$, Proposition \ref{prop:smp} applies to conclude $|\Phi| > 0$ a.e. in $B_1^c$.

\vspace{.4cm}

\noindent \textit{The uniqueness:} Let $\phi$ be a positive eigenfunction corresponding to $\lambda_1$. Suppose by contradiction that there exists another principal eigenvalue $\mu_1 > \lambda_1$ with corresponding positive eigenfunction $\psi$. Then by definition,
\begin{equation}
\label{phi-eqn}
\int_{B_1^c} |\nabla \phi|^{p-2}\nabla \phi \cdot \nabla v + \beta\int_{\partial B_1} |\phi|^{p-2}\phi v = \lambda_1 \int_{B_1^c} g|\phi|^{p-2}\phi v,
\end{equation}

\begin{equation}
\label{psi-eqn}
    \int_{B_1^c} |\nabla \psi|^{p-2}\nabla \psi \cdot \nabla v + \beta\int_{\partial B_1} |\psi|^{p-2}\psi v = \mu_1 \int_{B_1^c} g|\psi|^{p-2}\psi v
\end{equation}
for all $v \in W^{1,p}(B_1^c)$. Using Green's formula and taking $v = \phi$, we get
\[
\int_{B_1^c} |\nabla \phi|^p = -\int_{B_1^c} \text{div}(|\nabla \phi|^{p-2}\nabla \phi)\phi + \int_{\partial B_1} |\nabla \phi|^{p-2}\frac{\partial \phi}{\partial \nu}\phi.
\]
Using $v=\phi $ as the test function in \eqref{phi-eqn} we get
\[
\int_{B_1^c} |\nabla \phi|^p = \lambda_1 \int_{B_1^c} g|\phi|^p - \beta \int_{\partial B_1} |\phi|^p.
\]
Now consider the integral
\[
\int_{B_1^c} \nabla(\frac{\phi^p}{\psi^{p-1}}) \cdot |\nabla \psi|^{p-2}\nabla \psi.
\]
By Green's formula:
\[
\int_{B_1^c} \nabla(\frac{\phi^p}{\psi^{p-1}}) \cdot |\nabla \psi|^{p-2}\nabla \psi = -\int_{B_1^c} \text{div}(|\nabla \psi|^{p-2}\nabla \psi)\frac{\phi^p}{\psi^{p-1}} + \int_{\partial B_1} |\nabla \psi|^{p-2}\frac{\partial \psi}{\partial \nu}\frac{\phi^p}{\psi^{p-1}}.
\]
Using $v=\frac{\phi^p}{\psi^{p-1}} $ as the test function in \eqref{psi-eqn}:
\[
\int_{B_1^c} \nabla(\frac{\phi^p}{\psi^{p-1}}) \cdot |\nabla \psi|^{p-2}\nabla \psi = \mu_1 \int_{B_1^c} g|\phi|^p - \beta \int_{\partial B_1} |\phi|^p.
\]
Let $u,v \in  W^{1,p}(B_1^c)$  with $v > 0$ a.e. and let
\[
L(u,v) = |\nabla u|^p - p|u|^{p-1}|\nabla v|^{p-2}\nabla u \cdot \nabla(v)/(v)^{p-1} + (p-1)|u|^p|\nabla v|^p/(v)^p,
\]
and
\[
R(u,v) = |\nabla u|^p - \nabla(u^p/v^{p-1}) \cdot |\nabla v|^{p-2}\nabla v,
\]
By Picone's identity, for nonnegative functions $u,v \in W^{1,p}(B_1^c)$ with $v > 0$ a.e., we have
$L(u,v) = R(u,v) \geq 0$ with equality if and only if $u = cv$ for some constant $c$. Therefore:
\[
0\leq \int_{B_1^c} L(\phi,\psi) = \int_{B_1^c} R(\phi,\psi) = (\lambda_1 - \mu_1)\int_{B_1^c} g|\phi|^p.
\]
By assumption $\lambda_1 < \mu_1$, and since $\phi \in N_g$ we have $\int_{B_1^c} g|\phi|^p = 1$. This gives
\[
0 \leq \int_{B_1^c} L(\phi,\psi) = (\lambda_1 - \mu_1)\int_{B_1^c} g|\phi|^p = \lambda_1 - \mu_1 < 0,
\]
a contradiction. Therefore, $\lambda_1$ is the unique principal eigenvalue.

\vspace{.4cm}

\noindent\textbf{(b)} \textit{Boundedness:} 
Our proof employs a Moser-type iteration, adapted to our setting, following the approach of \cite{dr95}.
The weak formulation of problem \eqref{eq:main-intro} gives
\[
\int_{B_1^c} |\nabla\phi|^{p-2}\nabla\phi \cdot \nabla v \, dx + \beta\int_{\partial B_1} |\phi|^{p-2}\phi v \, dS = \lambda\int_{B_1^c} g|\phi|^{p-2}\phi v \, dx
\]
for all $v \in W^{1,p}(B_1^c)$. Taking $v = |\phi|^{(k-1)p}\phi$ where $k > 0$, we obtain
\[
((k-1)p + p)\int_{B_1^c} |\nabla\phi|^p|\phi|^{(k-1)p} \, dx + \beta\int_{\partial B_1} |\phi|^{kp} \, dS = \lambda\int_{B_1^c} g|\phi|^{kp} \, dx.
\]
For the right-hand side term, using Lemma \ref{lem:hardy-sobolev} we have,
\[
\left|\int_{B_1^c} g|\phi|^{kp} \, dx\right| \leq C\|g\|_{(N/p,\infty)} \int_{B_1^c} |\nabla(|\phi|^{kp/p})|^p \, dx.
\]
Since $\beta > 0$, the boundary term is nonnegative. Setting $w = |\phi|^{kp/p}$ and using the chain rule,
\[
|\nabla w|^p = (kp/p)^p |\phi|^{(k-1)p}|\nabla\phi|^p,
\]
we obtain
\[
\int_{B_1^c} |\nabla w|^p \, dx \leq C_1(k)\|w\|_{p^*}^p.
\]
By the Sobolev embedding for exterior domain,
\[
\|w\|_{p^*}^p \leq C_2\left(\int_{B_1^c} |\nabla w|^p \, dx + \int_{B_1^c} |w|^p \, dx\right).
\]
Choosing $k_1$ such that $(k_1+1)p = p^*$ and defining the iteration sequence $\{k_n\}$ by $(k_n + 1) = (p^*/p)^n$, we obtain
\[
\|\phi\|_{r_{n+1}} \leq C_4^{\frac{1}{k_n+1}}\|\phi\|_{r_n}
\]
where $r_n = (k_n + 1)p$. Iterating these inequalities and taking $n \to \infty$ yields
\[
\|\phi\|_{L^\infty(B_1^c)} \leq C\|\phi\|_{L^{p^*}(B_1^c)} \leq C\|\phi\|_{W^{1,p}(B_1^c)}.
\]

\vspace{.4cm}

\noindent\textbf{(c)} \textit{Higher Regularity:} 
We regard the equation $-\Delta_p \phi = \lambda g |\phi|^{p-2}\phi$ as an inhomogeneous $p$-Laplace equation of the form
$$-\Delta_p \phi = \mathcal{L}(x)$$ where $\mathcal{L}(x) := \lambda g(x)|\phi(x)|^{p-2}\phi(x)$. 
From part (b), we know that $\phi \in L^\infty(B_1^c)$, which implies
$$|\mathcal{L}(x)| \leq \lambda \|\phi\|_{L^\infty}^{p-1} |g(x)|.$$
Consequently, if $g \in L^q_{loc}(B_1^c)$ for some $q \in [1,\infty]$, then 
$\mathcal{L} \in L^q_{loc}(B_1^c)$ with
$$\|\mathcal{L}\|_{L^q(K)} \leq \lambda \|\phi\|_{L^\infty}^{p-1} \|g\|_{L^q(K)}$$
for any compact $K \subset B_1^c$.
To apply the $C^{1,\alpha}$ regularity results of DiBenedetto 
\cite{dibenedetto83}, the source term must belong to $L^q_{loc}(B_1^c)$ 
with $q > \frac{Np}{p-1}$ (see Remark on p.~829 of \cite{dibenedetto83}). 
Under our additional assumption that $g \in L^q_{loc}(B_1^c)$ for some 
$q > \frac{Np}{p-1}$, this condition is satisfied. 
Therefore, by \cite[Theorems 1 and 2]{dibenedetto83}, we conclude that 
$\phi \in C^{1,\alpha}_{loc}(B_1^c)$ for some $\alpha \in (0,1)$ depending only on $N$, $p$, and local bounds on $\phi$ and $g$.

For boundary regularity, we need to verify condition (0.6) of Lieberman \cite{lieberman88} for our Robin boundary term $\phi(x,z) = -\beta|z|^{p-2}z$. Since $\phi \in L^\infty(B_1^c)$, there exists $M > 0$ such that $|z| \leq M$ for all $z$ in consideration. For any $z,w$ with $|z|,|w| \leq M$, using mean value theorem,
\[
|\phi(x,z) - \phi(y,w)| = \beta||z|^{p-2}z - |w|^{p-2}w| \leq \beta C(p,M)|z-w|^{\min\{1,p-1\}}
\] 
where $C(p,M)$ depends only on $p$ and $M$. Since $\beta$ is constant, the $x,y$ dependence vanishes and we have
\[
|\phi(x,z) - \phi(y,w)| \leq \Phi(|x-y|^{\alpha} + |z-w|^{\alpha})
\]
where $\alpha = \min\{1,p-1\}$ and $\Phi = \beta C(p,M)$. This verifies condition (0.6) of Lieberman. Therefore, by Theorem 2 of Lieberman \cite{lieberman88}, we obtain $\phi \in C^{1,\alpha}$ in a neighborhood of $\partial B_1$. The combination of interior and boundary estimates gives $\phi \in C^{1,\alpha}_{loc}(\overline{B_1^c})$.

\vspace{.4cm}

\noindent\textbf{(d)} \textit{Simplicity:} Let $(\lambda_1,\phi)$ and $(\lambda_1,\psi)$ be two positive solutions of \eqref{eq:main-intro}. By Lemma \ref{lem:zeroset}, both sets $Z_1 = \{x \in B_1^c : \phi(x) = 0\}$ and $Z_2 = \{x \in B_1^c : \psi(x) = 0\}$ have zero $W^{1,p}$-capacity. Consequently, their union $Z = Z_1 \cup Z_2$ also has zero $W^{1,p}$-capacity. For $x \in B_1^c \setminus Z$, from Proposition 3.1 in Lucia and Prashanth \cite{lp06}, we have:
\[
\nabla \phi\psi - \phi\nabla \psi = 0 \quad \text{H}^N\text{-a.e. in } B_1^c \setminus Z.
\]
It shows that $\nabla(\log \phi - \log \psi) = 0$ in $B_1^c \setminus Z$, which implies $\phi = c\psi$ for some constant $c$ in any connected component of $B_1^c \setminus Z$.
We need to establish that $B_1^c \setminus Z$ is indeed connected. Since $Z$ has zero $W^{1,p}$-capacity, this follows from Lemma 2.46 in Heinonen et al. \cite{hkm93}, which indicates that removing a set of zero capacity cannot disconnect a domain. Alternatively, as noted by Kawohl (acknowledged in Lucia and Prashanth \cite{lp06}), if $B_1^c \setminus Z$ had multiple connected components, we could define $\phi$ to be a positive multiple of $\psi$ in one component and a negative multiple in another. This would give a principal eigenfunction that changes sign, contradicting Proposition \ref{prop:smp}. Therefore, $B_1^c \setminus Z$ must be connected, and consequently $\phi = c\psi$ throughout $B_1^c \setminus Z$ for some constant $c$.

Since $\partial B_1$ is $C^2$, we have the continuous embedding $W^{1,p}(B_1^c) \hookrightarrow W^{1-1/p,p}(\partial B_1)$. Since $Z$ has zero $W^{1,p}$-capacity, its intersection with $\partial B_1$ has zero $(N-1)$-dimensional Hausdorff measure. Hence, almost every point on $\partial B_1$ is not in $Z$. For any $x_0 \in \partial B_1 \setminus Z$, the density of $B_1^c \setminus Z$ near $x_0$ ensures we can find a sequence $\{x_n\} \subset B_1^c \setminus Z$ with $x_n \to x_0$ such that:
\[
\phi(x_n) = c\psi(x_n) \quad \text{and} \quad \nabla\phi(x_n) = c\nabla\psi(x_n).
\]
Taking limits and using the $C^{1,\alpha}$ regularity established in part (c):
\[
\phi(x_0) = c\psi(x_0) \quad \text{and} \quad \nabla\phi(x_0) = c\nabla\psi(x_0).
\]
On $\partial B_1$, the Robin boundary conditions give:
\[
|\nabla\phi|^{p-2}\frac{\partial\phi}{\partial\nu} + \beta|\phi|^{p-2}\phi = 0
\quad \text{and} \quad
|\nabla\psi|^{p-2}\frac{\partial\psi}{\partial\nu} + \beta|\psi|^{p-2}\psi = 0.
\]
Substituting $\phi = c\psi$ and $\nabla\phi = c\nabla\psi$ in the first equation:
\[
|c\nabla\psi|^{p-2}c\frac{\partial\psi}{\partial\nu} + \beta|c\psi|^{p-2}c\psi = 0,
\]
which simplifies to
\[
|c|^{p-2}c\left(|\nabla\psi|^{p-2}\frac{\partial\psi}{\partial\nu} + \beta|\psi|^{p-2}\psi\right) = 0.
\]
Since both $\phi$ and $\psi$ are normalized principal eigenfunctions satisfying $\int_{B_1^c} g|\phi|^p = \int_{B_1^c} g|\psi|^p = 1$, and we have established $\phi = c\psi$ almost everywhere in $B_1^c$, we obtain:
$$1 = \int_{B_1^c} g|\phi|^p = |c|^p \int_{B_1^c} g|\psi|^p = |c|^p.$$
Therefore $|c| = 1$. Since both $\phi$ and $\psi$ are positive almost everywhere in $B_1^c$ due to strong maximum principle (Proposition \ref{prop:smp}), we conclude $c = 1$. Therefore, $\phi = \psi$ in $B_1^c \cup \partial B_1$ except for a set of zero $W^{1,p}$-capacity, proving that $\lambda_1$ is simple.

\vspace{.4cm}

\noindent\textbf{(e)} \textit{Isolatedness:} The isolation argument follows the same Palais-Smale approach as in Anoop and Biswas \cite{anoop19}. The key steps are: (i) assuming $\lambda_1$ is not isolated gives a sequence of eigenvalues $\lambda_n \to \lambda_1$, (ii) the corresponding eigenfunctions $\phi_n$ satisfy the Palais-Smale condition, (iii) extracting a convergent subsequence that must converge to $\pm|\Phi|$ by simplicity, and (iv) deriving a contradiction using the positivity of $|\Phi|$ on sets of positive measure. 

\vspace{.4cm}

Having established the fundamental properties of the principal eigenfunction - existence, uniqueness, simplicity, regularity, and isolation of the principal eigenvalue - we will now investigate  the change in the principal eigenvalue $\lambda_1$ and the corresponding eigenfunctions as the Robin parameter $\beta$ is varied. While our primary focus in this paper is on the fixed-$\beta$ analysis and the subsequent asymptotic and gradient estimates, the variational framework naturally yields parameter-dependence results that provide valuable insight into the transition between classical Neumann ($\beta \to 0$) and Dirichlet ($\beta \to \infty$) boundary conditions.The following remark, summarizes these supplementary results.

\begin{remark}[Dependence on the Robin Parameter]\label{rem:beta_dependence}
Following the approach of Dr\'abek and Rasouli in \cite{dr10} where they studied the Robin eigenvalue problem on domains of finite measure, we obtain the following properties:

\noindent\textbf{(i) Basic Properties:} The function $\beta \mapsto \lambda_1(\beta)$ is concave, strictly increasing, and continuously differentiable on $(0,\infty)$. Moreover, the derivative admits the explicit formula
\begin{equation}\label{eq:derivative_formula}
\frac{d\lambda_1}{d\beta} = \int_{\partial B_1} \phi_1(\beta)^p \, dH^{N-1},
\end{equation}
which follows from the normalization $\int_{B_1^c} g|\phi_1(\beta)|^p = 1$ and variational characterization of $\lambda_1(\beta)$. The strict positivity of the right-hand side (ensured by V\'azquez's maximum principle \cite{vaz84} and $\beta > 0$) guarantees monotonicity.

\noindent\textbf{(ii) Eigenfunction Continuity:} The eigenfunction $\beta \mapsto \phi_1(\beta)$ is continuous as a map from $(0,\infty)$ into $W^{1,p}(B_1^c)$. This follows from the compactness of the weight functional $G(u) = \int_{B_1^c} g|u|^p$ (Proposition \ref{ref:compactness}) combined with weak lower semicontinuity arguments and the uniform convexity of $W^{1,p}(B_1^c)$.

\noindent\textbf{(iii) Dirichlet Limit ($\beta \to \infty$):} As $\beta \to \infty$, we have
\[
\lambda_1(\beta) \to \lambda_1^D \quad \text{and} \quad \phi_1(\beta) \to \phi_1^D \quad \text{strongly in } W^{1,p}(B_1^c),
\]
where $\lambda_1^D$ and $\phi_1^D$ denote the principal eigenvalue and eigenfunction for the Dirichlet problem $u|_{\partial B_1} = 0$. From the variational characterization of $\lambda_1(\beta)$ and $ \lambda_1^D$, it is easy to see that $\lambda_1(\beta) \leq \lambda_1^D$. The convergence follows from: (a) the monotonicity and boundedness  for all $\beta > 0$, (b) extracting a weak limit of the bounded sequence $\{\phi_1(\beta)\}$ as $\beta \to \infty$, and (c) verifying that the limit satisfies the Dirichlet boundary condition. Since $\frac{d\lambda_1}{d\beta}<0$, it also follows that $\lambda_1(\beta) < \lambda_1^D$ for all  $\beta$.

\noindent\textbf{(iv) Neumann Limit ($\beta \to 0^+$):} The limiting behavior as $\beta \to 0^+$ exhibits fundamentally different characteristics in exterior domains compared to bounded domains. In bounded domains, Neumann eigenfunctions are constants (with eigenvalue zero), but this is \emph{not} true for exterior domains. By Theorem 1.1 of Anoop and Biswas \cite{anoop19}, the Neumann principal eigenvalue on $B_1^c$ with weight $g \in F^{N/p}$ is strictly positive: $\lambda_1^N > 0$. Consequently, the Neumann eigenfunction $\phi_1^N$ is non-trivial with $\int_{B_1^c} |\nabla\phi_1^N|^p = \lambda_1^N > 0$. 
Using this result, we establish the Neumann limit:
\[
\lambda_1(\beta) \to \lambda_1^N \quad \text{and} \quad \phi_1(\beta) \to \phi_1^N \quad \text{strongly in } W^{1,p}(B_1^c) \quad \text{as } \beta \to 0^+.
\]
Since $\lambda_1(\beta) = \int_{B_1^c} |\nabla\phi_1(\beta)|^p + \beta\int_{\partial B_1} |\phi_1(\beta)|^p$ and the boundary term is nonnegative, we have $\lambda_1(\beta) \geq \int_{B_1^c} |\nabla\phi_1(\beta)|^p \geq \lambda_1^N$ by the variational characterization of the Neumann problem, yielding $\liminf_{\beta\to 0^+} \lambda_1(\beta) \geq \lambda_1^N$. 
For the upper bound, we use the Neumann eigenfunction $\phi_1^N \in N_g$ as a test function. By the trace theorem, $\int_{\partial B_1} |\phi_1^N|^p =: C < \infty$. The variational characterization gives
$$\lambda_1(\beta) \leq \int_{B_1^c} |\nabla\phi_1^N|^p + \beta\int_{\partial B_1} |\phi_1^N|^p = \lambda_1^N + \beta C.$$
Taking the limit as $\beta \to 0^+$ yields $\limsup_{\beta\to 0^+} \lambda_1(\beta) \leq \lambda_1^N$. The squeeze $\lambda_1^N \leq \liminf_{\beta\to 0^+} \lambda_1(\beta) \leq \limsup_{\beta\to 0^+} \lambda_1(\beta) \leq \lambda_1^N$ forces $\lambda_1(\beta) \to \lambda_1^N$.
The compactness of $G$ (Proposition \ref{ref:compactness}) combined with weak lower semicontinuity identifies the weak limit of $\{\phi_1(\beta)\}$ as $\phi_1^N$, and the uniform convexity of $L^p$ spaces upgrades weak convergence to strong convergence.
\end{remark}

Next, we turn our attention to obtaining estimates for the eigenfunctions near the boundary and for large $|x|$.

\subsection{Proof of Theorem \ref{thm:asymptotic-intro}}

The proof of Theorem \ref{thm:asymptotic-intro} combines several complementary techniques. For the general decay estimates in part (a), we follow the comparison principle approach of Chhetri and Dr\'{a}bek \cite{cd14}, utilizing Serrin's local estimates and critical exponent analysis to establish universal far-field decay rates independent of the Robin parameter. The radial symmetry in part (b) follows from rotational invariance of the variational characterization. For the boundary estimates in part (c), we exploit the Robin boundary condition to derive the explicit relationship $\phi'(1) = \beta^{1/(p-1)}\phi(1)$, then apply the $C^{1,\alpha}$ regularity from Theorem \ref{thm:principal-intro} to extend this near-boundary behavior, revealing the characteristic $\beta$-dependence of the solutions near the boundary.

\vspace{.4cm}

\noindent\textbf{(a)} \textit{General Decay at Infinity:} 
Since $g \in F^{N/p} \subset L^{N/p,\infty}(B_1^c)$ and $\phi \in L^\infty(B_1^c)$ by Theorem \ref{thm:principal-intro}, we can apply Serrin's local estimates \cite{serrin64}. For any $x$ with $|x|$ sufficiently large:
\[
\|\phi\|_{L^\infty(B_1(x))} \leq C\left(\|\phi\|_{L^{p^*}(B_2(x))} + \|\lambda_1 g\phi^{p-1}\|_{L^{\gamma_1}(B_2(x))}\right),
\]
where $\gamma_1 > N/p$. The weight decay condition $g(x) \leq C_0|x|^{-l}$ with $l > p$ ensures that as $|x| \to \infty$, the right-hand side vanishes, proving $\phi(x) \to 0$ uniformly.
The eigenfunction satisfies $-\Delta_p \phi = \lambda_1 g\phi^{p-1} \geq \lambda_1 c|x|^{-l}\phi^{p-1}$ for some $c > 0$ when $|x|$ is sufficiently large. By using \cite[Proposition 2.6]{veron01}, we get:
\[
\phi(x) \geq \frac{C_1}{|x|^{(N-p)/(p-1)}} \quad \text{for } |x| \geq r_0.
\]
Setting $f(x,u) = \lambda_1 g(x)u^{p-1}$, we have $f(x,u) \leq \lambda_1 C_0|x|^{-l}u^{p-1}$. The critical exponent analysis from \cite[Theorem 4]{avila08} applies: since $l > p$, we have $q = p-1 > q_l = (N-l)(p-1)/(N-p)$. This yields:
\[
\phi(x) \leq \frac{C_2}{|x|^{(N-p)/(p-1)}} \quad \text{for } |x| \geq r_0.
\]

\vspace{0.3cm}

\noindent\textbf{(b)} \textit{Radial Symmetry:} The variational characterization of $\lambda_1$ is rotationally invariant under our assumptions. Since $g(x) = g(|x|)$ and $B_1^c$ is rotationally symmetric, both functionals
$$J(u) = \int_{B_1^c} |\nabla u|^p + \beta\int_{\partial B_1} |u|^p \quad \text{and} \quad G(u) = \int_{B_1^c} g|u|^p$$
are invariant under rotations about the origin. By the simplicity of $\lambda_1$ (Theorem \ref{thm:principal-intro}), if $\phi$ is a principal eigenfunction, then so is $\phi \circ R^{-1}$ for any rotation $R$. Uniqueness implies $\phi \circ R^{-1} = \phi$, which forces $\phi(x)$ to be radial.

\vspace{0.3cm}

\noindent\textbf{(c)} \textit{Boundary Estimates:} 
From part (b), the eigenfunction $\phi$ is radially symmetric, i.e., $\phi(x) = \varphi(|x|)$ for all $x \in B_1^c$ for some $\varphi:[1,\infty)\to\mathbb{R}$. Under radial symmetry, the gradient is
$$\nabla \phi(x) = \varphi'(|x|) \frac{x}{|x|}.$$
On the boundary $\partial B_1$ where $|x| = 1$, this becomes
$\nabla \phi(x) = \varphi'(1)  x.$
Since the outward unit normal at $x \in \partial B_1$ is $\nu(x) = -x$, the normal derivative is
$$\frac{\partial \phi}{\partial \nu} = \nabla \phi \cdot \nu = \varphi'(1)~ x \cdot (-x) = -\varphi'(1).$$
From the Robin boundary condition and using \cite[Theorem 5]{vaz84} we conclude that $\frac{\partial \phi}{ \partial \nu}<0 $. So, we have $\varphi'(1) > 0$. The gradient magnitude is
$$|\nabla \phi(x)| = |\varphi'(1) \cdot x| = |\varphi'(1)| \quad \text{for } x \in \partial B_1.$$
Substituting into the Robin boundary condition $|\nabla \phi|^{p-2}\frac{\partial \phi}{\partial \nu} + \beta|\phi|^{p-2}\phi = 0$:
\begin{align*}
    |\varphi'(1)|^{p-2}(-\varphi'(1)) + \beta|\varphi(1)|^{p-2}\varphi(1) = 0. \label{rr_general}
\end{align*}
Since we have $\varphi'(1) >0$, this simplifies to
$$\varphi'(1) = \beta^{1/(p-1)} \varphi(1).$$
By the $C^{1,\alpha}$ regularity from Theorem \ref{thm:principal-intro}, for $r = 1 + t$ with small $t > 0$:
$$\varphi(1+t) = \varphi(1) + \varphi'(1)t + O(t^{1+\alpha}) = \varphi(1) + \beta^{1/(p-1)} \varphi(1) \cdot t + O(t^{1+\alpha})$$
$$= \varphi(1)(1 + \beta^{1/(p-1)} t) + O(t^{1+\alpha}).$$
The $C^{1,\alpha}$ regularity implies there exist constants $L > 0$ and $\delta > 0$ such that for all $0 < t < \delta$:
$$|\varphi(1+t) - \varphi(1) - \varphi'(1)t| \leq L t^{1+\alpha}.$$
Substituting our values:
$$|\varphi(1+t) - \varphi(1)(1 + \beta^{1/(p-1)} t)| \leq L t^{1+\alpha}.$$
Rearranging, we get
$$\varphi(1)(1 + \beta^{1/(p-1)} t) - L t^{1+\alpha} \leq \varphi(1+t) \leq \varphi(1)(1 + \beta^{1/(p-1)} t) + L t^{1+\alpha}.$$
Since $\phi(x) = \varphi(|x|) = \varphi(1+t)$ where $t = |x| - 1$, setting $K = L$ establishes \eqref{eq:boundary_bounds}.

\begin{remark}\label{rem:limiting_behavior}
The decay estimates in Theorem \ref{thm:asymptotic-intro} exhibit the expected limiting behavior as the Robin parameter $\beta$ varies, providing a unifying framework that interpolates between classical boundary conditions:

\noindent\textbf{(i) Dirichlet Limit ($\beta \to \infty$):} The rescaled function $\Psi(x) := \frac{\phi(x)-\varphi(1)}{\varphi(1)\beta^{1/(p-1)}}$ reveals the connection to Dirichlet problems. From the boundary decay estimate \eqref{eq:boundary_bounds}, we have
$$\Psi(x) = \frac{\phi(x) - \varphi(1)}{\varphi(1)\beta^{1/(p-1)}} \approx \frac{\varphi(1)\beta^{1/(p-1)} t - K t^{1+\alpha}}{\varphi(1)\beta^{1/(p-1)}} = t - \frac{K t^{1+\alpha}}{\beta^{1/(p-1)}}$$
where $t = |x| - 1 = \mathrm{dist}(x,\partial B_1)$. As $\beta \to \infty$, we obtain $\psi(x) \to t$ with $\psi = 0$ on $\partial B_1$, exhibiting the characteristic linear behavior $\psi(x) \approx \mathrm{dist}(x,\partial B_1)$ of Dirichlet eigenfunctions. This is consistent with the results of Chhetri and Dr\'{a}bek \cite{cd14} for the Dirichlet case.

\noindent\textbf{(ii) Neumann Limit ($\beta \to 0$):} As $\beta \to 0$, the Robin boundary condition becomes $\partial\phi/\partial\nu = 0$. The boundary decay estimate reduces to $\phi(x) \approx \varphi(1) + O(t^{1+\alpha})$, showing that $\phi$ is approximately constant near the boundary, which is precisely the expected behavior for Neumann conditions where the normal derivative vanishes.
\end{remark}

\subsection{Proof of Theorem \ref{thm:gradient-near-intro}}
For the near-boundary estimates in part (a), we exploit the direct relationship $|\nabla\phi| = \beta^{1/(p-1)}\phi(1)$ on $\partial B_1$ derived from the Robin boundary condition, then extend these bounds using the $C^{1,\alpha}$ regularity from Theorem \ref{thm:principal-intro}. The far-field estimates in part (b) establish the existence and uniqueness of the critical point $r_* > 1$ where $\phi'(r_*) = 0$ by analyzing the monotonicity of $F(r) = r^{N-1}|\phi'(r)|^{p-2}\phi'(r)$. 


\vspace{.4cm}

\noindent\textbf{(a)} \textit{Gradient estimate near boundary:} 
Following the same analysis as in the proof of Theorem \ref{thm:asymptotic-intro} part (c), we obtain from the Robin boundary condition and radial symmetry:
 $$|\nabla \phi(x)| = \beta^{1/(p-1)}\varphi(x) \quad \text{for all } x \in \partial B_1$$
Due to radial symmetry, $\phi(x) = \varphi(1)$ is constant on $\partial B_1$, so:
$$|\nabla \phi(x)| = \beta^{1/(p-1)}\varphi(1) \quad \text{for all } x \in \partial B_1$$
The $C^{1,\alpha}$ regularity from Theorem \ref{thm:principal-intro} extends these boundary estimates to nearby interior points. Specifically, there exists $\Lambda > 0$ such that $$|\nabla \phi(x) - \nabla \phi(y)| \leq \Lambda |x - y|^{\alpha} ~~ \forall x, y \in B_1^c.$$ For $x \in B_1^c$ near $\partial B_1$, let $x_0 \in \partial B_1$ be the closest boundary point, so that $|x - x_0| = \mathrm{dist}(x,\partial B_1)$. 
For the lower bound, the triangle inequality and Hölder continuity give
$$|\nabla \phi(x)| \geq |\nabla \phi(x_0)| - \Lambda\cdot\mathrm{dist}(x,\partial B_1)^{\alpha} \geq \beta^{1/(p-1)}\varphi(1) - \Lambda\cdot\mathrm{dist}(x,\partial B_1)^{\alpha}.$$
Choosing $\delta = \left(\frac{\beta^{1/(p-1)}\varphi(1)}{2\Lambda}\right)^{1/\alpha}$ ensures that for all $x$ with $\mathrm{dist}(x,\partial B_1) < \delta$, we have $\Lambda\cdot\mathrm{dist}(x,\partial B_1)^{\alpha} < \frac{\beta^{1/(p-1)}\varphi(1)}{2}$, yielding 
$$|\nabla \phi(x)| \geq \frac{1}{2}\beta^{1/(p-1)}\varphi(1).$$
For the upper bound, the reverse triangle inequality gives
$$|\nabla \phi(x)| \leq |\nabla \phi(x_0)| + \Lambda\cdot\mathrm{dist}(x,\partial B_1)^{\alpha} \leq M_0\beta^{1/(p-1)}\varphi(1) + \Lambda\cdot\mathrm{dist}(x,\partial B_1)^{\alpha}.$$
Factoring yields $|\nabla \phi(x)| \leq M_0\beta^{1/(p-1)}\varphi(1)\left(1 + \frac{\Lambda}{M_0\beta^{1/(p-1)}\varphi(1)}\mathrm{dist}(x,\partial B_1)^{\alpha}\right)$. Setting $m = \frac{1}{2}$, $M = M_0$, and $C = \frac{\Lambda}{M_0\beta^{1/(p-1)}\varphi(1)}$ completes the proof.

\vspace{.4cm}

\noindent\textbf{(b)} \textit{Gradient estimate near infinity:} 
Since the problem has radial symmetry, we write $\phi(x) = \varphi(r)$ where $r = |x|$. 
By the regularity arguments given in \cite[p.110]{dk12}, the eigenfunction $\varphi$ satisfies the following ODE in $(1,\infty)$: 
\begin{equation}
    (r^{N-1}|\varphi'(r)|^{p-2}\varphi'(r))' + \lambda_1 r^{N-1}g(r)|\varphi(r)|^{p-2}\varphi(r) = 0. \label{ge1}
\end{equation}
We first establish the monotonicity structure of $\phi$. From Theorem \ref{thm:asymptotic-intro} part (b), we have $\varphi'(1) = \beta^{1/(p-1)}\varphi(1) > 0$. Since $\varphi(r) \to 0$ as $r \to \infty$ while $\varphi(r) > 0$ for all $r > 1$ by Proposition \ref{prop:smp}, the derivative $\phi'(r)$ must become negative at some point. Otherwise, $\phi$ would be non-decreasing and could not decay to zero.
To show uniqueness of the critical point, we define $F(r) := r^{N-1}|\varphi'(r)|^{p-2}\varphi'(r)$. Since $\lambda_1 > 0$, $g(r) > 0$, and $\varphi(r) > 0$, from \eqref{ge1} we have:
$$F'(r) = -\lambda_1 r^{N-1}g(r)|\varphi(r)|^{p-2}\varphi(r) < 0 \quad \text{for all } r > 1.$$
If there exist two distinct critical points $a, b > 1$ with $a < b$ such that $\varphi'(a) = 0 = \varphi'(b)$, then $F(a) = F(b) = 0$. By Rolle's theorem, there exists $c \in (a,b)$ such that $F'(c) = 0$, contradicting $F'(r) < 0$ for all $r > 1$. Therefore, there exists a unique $r_* > 1$ such that $\varphi'(r_*) = 0$, with $\phi'(r) > 0$ for $1 < r < r_*$ and $\varphi'(r) < 0$ for $r > r_*$.

Since $\varphi > 0$ by Proposition \ref{prop:smp} and $\varphi'(r) < 0$ for $r > r_*$, we have $|\varphi'(r)|^{p-2}\varphi'(r) = -|\varphi'(r)|^{p-1}$ for $r > r_*$. Rearranging \eqref{ge1} and integrating from $r_*$ to $r$ (with $r > r_*$) yields:
$$r^{N-1}|\varphi'(r)|^{p-1} = |\varphi'(r_*)|^{p-1} + \lambda_1\int_{r_*}^r s^{N-1}g(s)\varphi^{p-1}(s)ds.$$
Since $\phi'(r_*) = 0$, this simplifies to:
\begin{equation}
    r^{N-1}|\varphi'(r)|^{p-1} = \lambda_1\int_{r_*}^r s^{N-1}g(s)\varphi^{p-1}(s)ds.\label{ge2}
\end{equation}
From Theorem \ref{thm:asymptotic-intro}, for sufficiently large $s \geq r_0$, the eigenfunction satisfies the asymptotic decay estimate 
$$\varphi(s) \geq \frac{C_{\text{low}}}{s^{(N-p)/(p-1)}} $$ 
for some positive constant $C_{\text{low}}$. Since $g(x) > 0$ for all $x \in B_1^c$ and $g$ is continuous, on the compact interval $[r_*, 2r_*]$ we have $g_{\min} := \min_{s \in [r_*, 2r_*]} g(s) > 0$. For $r \geq 2r_*$, we estimate the integral term from below:
\begin{align*}
\int_{r_*}^r s^{N-1}g(s)\varphi^{p-1}(s)ds &\geq \int_{r_*}^{2r_*} s^{N-1}g(s)\varphi^{p-1}(s)ds \\
&\geq \int_{r_*}^{2r_*} s^{N-1} \cdot g_{\min} \cdot \left(\frac{C_{\text{low}}}{s^{(N-p)/(p-1)}}\right)^{p-1} ds \\
&= g_{\min} C_{\text{low}}^{p-1} \int_{r_*}^{2r_*} s^{p-1} ds \\
&= g_{\min} C_{\text{low}}^{p-1} \frac{(2r_*)^p - r_*^p}{p} =: I_0 > 0.
\end{align*}
Therefore, for $r \geq 2r_*$, we have $r^{N-1}|\varphi'(r)|^{p-1} \geq \lambda_1 I_0$, which gives us:
$$|\varphi'(r)| \geq (\lambda_1 I_0)^{1/(p-1)} r^{-(N-1)/(p-1)}.$$
Setting $\hat{C}_1 := (\lambda_1 I_0)^{1/(p-1)}$, we obtain the desired lower bound:
$$|\nabla \phi(x)| \geq \hat{C}_1|x|^{-(N-1)/(p-1)} \quad \text{for all } x \in B_{2r_*}^c.$$
Note that $\hat{C}_1 > 0$ depends on the eigenvalue $\lambda_1$ and the weight function $g$, but is independent of the Robin parameter $\beta$.

For the upper bound, we start with (\ref{ge2}):
$$r^{N-1}|\varphi'(r)|^{p-1} = \lambda_1\int_{r_*}^r s^{N-1}g(s)\varphi^{p-1}(s)ds.$$
For the integral term, we apply the asymptotic decay estimate $\varphi(s)  \leq \frac{C_2}{s^{(N-p)/(p-1)}}$ from Theorem \ref{thm:asymptotic-intro} to obtain:
$$\int_{r_*}^r s^{N-1}g(s)\varphi^{p-1}(s)ds \leq C_2^{p-1} \int_{r_*}^r s^{p-1}g(s) ds.$$
The critical step involves applying Hölder's inequality with exponents $\frac{N}{p}$ and $\frac{N}{N-p}$, utilizing the $L^{N/p}$ structure of the weight function. As established by Chhetri-Dräbek, this yields:
$$\int_{r_*}^r s^{p-1}g(s) ds \leq \|g\|_{L^{N/p}(B_1^c)}^{p/N} C_{\text{geom}}(\log r)^{(N-p)/N},$$
where $C_{\text{geom}}$ is a geometric constant depending on $N$, $p$, and the domain parameters. Therefore:
$$\int_{r_*}^r s^{N-1}g(s)\varphi^{p-1}(s)ds \leq C^{*}(\log r)^{(N-p)/N}$$
where $C^{*} := C_2^{p-1} \|g\|_{L^{N/p}(B_1^c)}^{p/N} C_{\text{geom}}$.
Taking the $(p-1)$-th root we obtain
$$|\varphi'(r)| \leq \left[\lambda_1 C^{*}(\log r)^{(N-p)/N}\right]^{1/(p-1)} r^{-(N-1)/(p-1)}$$
$$\implies |\varphi'(r)| \leq \hat{C}_2 (\log r)^{(N-p)/(N(p-1))} r^{-(N-1)/(p-1)}$$
where the explicit constant is given by:
$$\hat{C}_2 = \left[\lambda_1 C_2^{p-1} \|g\|_{L^{N/p}(B_1^c)}^{p/N} C_{\text{geom}}\right]^{1/(p-1)}.$$

\vspace{.4cm}

Having established the existence, uniqueness, and gradient bounds near $r_*$, we now quantify the non-degeneracy of this critical point through a local growth estimate. This technical result will be essential for the unified gradient framework developed in Theorem \ref{thm:unified-intro}.

\begin{lemma}[]\label{lem:nondegeneracy}
Let $\varphi$ be the positive radial principal eigenfunction corresponding to eigenvalue $\lambda_1$ of problem \eqref{eq:main-intro}, and let $r_* = r_*(\beta) > 1$ denote the unique critical point where $\varphi'(r_*) = 0$ established in part (b). Then there exist constants $c_0 > 0$ and $\rho_0 > 0$ such that for all $r$ with $|r - r_*| < \rho_0$:
\begin{equation}\label{eq:nondegeneracy}
|\varphi'(r)| \geq c_0|r - r_*|^{1/(p-1)}
\end{equation}
where $c_0$ depends on $\beta$, $\lambda_1$, $g(r_*)$, $\phi(r_*)$, $r_*$, $N$, and $p$.
\end{lemma}

\begin{proof}
From the analysis in Theorem \ref{thm:gradient-near-intro}(b), the function $F(r) := r^{N-1}|\varphi'(r)|^{p-2}\varphi'(r)$ satisfies:
$$F(r_*) = 0 \quad \text{and} \quad F'(r) = -\lambda_1 r^{N-1}g(r)\varphi^{p-1}(r) < 0 \quad \text{for all } r > 1.$$
In particular:
$$F'(r_*) = -\lambda_1 r_*^{N-1}g(r_*)\varphi^{p-1}(r_*) < 0$$
Since $F$ is continuously differentiable and strictly decreasing with $F(r_*) = 0$, the Taylor expansion yields:
$$F(r) = F'(r_*)(r - r_*) + O((r-r_*)^2)$$
Therefore, for $\rho_0 > 0$ sufficiently small and $|r - r_*| < \rho_0$: 
$$ r^{N-1}|\varphi'(r)|^{p-1}=|F(r)| \geq \frac{|F'(r_*)|}{2}|r - r_*|$$
For $|r - r_*| < \rho_0 \leq r_*\left(2^{1/(N-1)} - 1\right)$, the continuity of $r \mapsto r^{N-1}$ ensures $r^{N-1} \leq 2r_*^{N-1}$. Thus:
$$|\varphi'(r)|^{p-1} \geq \frac{|F'(r_*)|}{2 \cdot 2r_*^{N-1}}|r - r_*| = \frac{\lambda_1 g(r_*)\varphi^{p-1}(r_*)}{4}|r - r_*|$$
Taking the $(p-1)$-th root and setting:
$$c_0 = \left(\frac{\lambda_1 g(r_*)\varphi^{p-1}(r_*)}{4}\right)^{1/(p-1)} > 0$$
yields the desired bound.
\end{proof}

\begin{remark}
The exponent $1/(p-1)$ is characteristic of the $p$-Laplacian operator. When $p = 2$, this reduces to the classical linear bound $|\varphi'(r)| \geq c_0|r - r_*|$ for the standard Laplacian. 
\end{remark}

\vspace{.4cm}

\subsection{Proof of Theorem \ref{thm:critical-intro}}
Having established the existence and uniqueness of $r_*$ in Theorem \ref{thm:gradient-near-intro}, we now derive explicit bounds for the boundary value $\varphi(1)$ and the critical point location $r_*$ in terms of the Robin parameter $\beta$. Both results exploit the fundamental identity arising from integrating the monotone quantity $F(r) = r^{N-1}|\varphi'(r)|^{p-2}\varphi'(r)$ from the boundary to the critical point.

\vspace{.4cm}

\noindent \textbf{(a)} \textit{Bounds on Boundary Value:} First we derive the upper bounds. We leverage the same framework developed in Theorem \ref{thm:gradient-near-intro}(b). Integrating $F'(r) = -\lambda_1 r^{N-1}g(r)\varphi(r)^{p-1}$ from 1 to $r_*$:
$$F(r_*) - F(1) = \int_1^{r_*} F'(s)\,ds.$$
Since $F(r_*) = 0$ (as $\varphi'(r_*) = 0$) and $F(1) = \beta\varphi(1)^{p-1}$ (from the Robin boundary condition $\varphi'(1) = \beta^{1/(p-1)}\varphi(1)$), we obtain the exact identity:
\begin{equation}\label{eq:fundamental_identity}
\beta\varphi(1)^{p-1} = \lambda_1 \int_1^{r_*} s^{N-1}g(s)\varphi(s)^{p-1}\,ds.
\end{equation}
By radial symmetry, the normalization condition becomes:
\begin{equation}\label{eq:radial_normalization}
\omega_{N-1}\int_1^\infty r^{N-1}g(r)\varphi(r)^p\,dr = 1.
\end{equation}
Since $\varphi$ is increasing on $[1, r_*]$, we have $\varphi(s) \geq \varphi(1)$ for all $s \in [1, r_*]$, giving $\varphi(s)^{p-1} = \varphi(s)^p/\varphi(s) \leq \varphi(s)^p/\varphi(1)$. Substituting into the fundamental identity \eqref{eq:fundamental_identity}:
\begin{align*}
\beta\varphi(1)^{p-1} &= \lambda_1\int_1^{r_*} s^{N-1}g(s)\varphi(s)^{p-1}\,ds \leq \frac{\lambda_1}{\varphi(1)}\int_1^{\infty} s^{N-1}g(s)\varphi(s)^p\,ds\\
&< \frac{\lambda_1}{\varphi(1)\omega_{N-1}}
\end{align*}
where the strict inequality follows  by positivity of the integrand.
Next we derive the lower bounds. The assumption $l > N$ ensures $I_g := \int_1^\infty r^{N-1}g(r)\,dr < \infty$.
At the boundary, the Robin condition gives $F(1) = \beta\varphi(1)^{p-1}$. For $r \in [1, r_*]$ where $\varphi' > 0$, the monotonicity $F(r) \leq F(1)$ yields:
$$\varphi'(r) \leq \frac{\beta^{1/(p-1)}\varphi(1)}{r^{(N-1)/(p-1)}}.$$
Integrating this bound from $1$ to $r$:
\begin{eqnarray*}
  \varphi(r)& \leq& \varphi(1)\left[1 + \beta^{1/(p-1)}\int_1^r s^{-(N-1)/(p-1)}\,ds\right]\\ 
  & <&\varphi(1)\left[1 + \beta^{1/(p-1)}\int_1^\infty s^{-(N-1)/(p-1)}\,ds\right].
\end{eqnarray*}
Therefore, we have:
\begin{equation}
    \varphi(r) \leq \varphi(1)\left[1 + \frac{p-1}{N-p}\beta^{1/(p-1)}\right] =: \varphi(1) \cdot K(\beta) \quad \text{for all } r \geq 1. \label{phi_r_bound}
\end{equation}
Applying this global bound to the normalization \eqref{eq:radial_normalization}:
$$1 = \omega_{N-1}\int_1^\infty r^{N-1}g(r)\varphi(r)^p\,dr \leq \omega_{N-1}[\varphi(1)K(\beta)]^p I_g.$$
Solving for $\varphi(1)$ and substituting the definition of $K(\beta)$ yields lower bounds with:
$$C_g = \left(\omega_{N-1}\int_1^\infty r^{N-1}g(r)\,dr\right)^{-1/p} > 0. \qedhere $$

\vspace{.4cm}

\noindent\textbf{(b)} \textit{Bounds on Critical Point:} For the upper bound, since $\varphi$ is increasing on $[1,r_*]$, we have $\varphi(s) \geq \varphi(1)$ and $g(s) \geq g_{\min}$:
$$\int_1^{r_*} s^{N-1}g(s)\varphi(s)^{p-1}\,ds \geq g_{\min}\varphi(1)^{p-1}\frac{r_*^N - 1}{N}.$$
Substituting into \eqref{eq:fundamental_identity} and rearranging yields the upper bound.
For the lower bound, using $\varphi(s) \leq \varphi(r_*)$ and $g(s) \leq g_{\max}$ we get
\begin{equation}
    \int_1^{r_*} s^{N-1}g(s)\varphi(s)^{p-1}\,ds \leq g_{\max}\ \varphi(r_*)^{p-1}\frac{r_*^N - 1}{N}. \label{lower_bound_for_r*}
\end{equation}
Now, from \eqref{phi_r_bound} we have
\begin{equation}
   \varphi(r_*)^{p-1} \leq \varphi(1)^{p-1}\left[1 + \frac{p-1}{N-p}\beta^{1/(p-1)}\right]^{p-1}. \label{lower_bound_for_r*_2}
\end{equation}
Using \eqref{lower_bound_for_r*} and \eqref{lower_bound_for_r*_2}, \eqref{eq:fundamental_identity} become
\begin{equation}
    \beta\varphi(1)^{p-1} \leq g_{\max}  \varphi(1)^{p-1}\left[1 + \frac{p-1}{N-p}\beta^{1/(p-1)}\right]^{p-1} \frac{r_*^N - 1}{N} \label{improved_lower_bound}
\end{equation}
After rearrangement of \eqref{improved_lower_bound}, we get
$$r_*^N \geq  1 + \frac{N\beta}{\lambda_1 g_{\max}\left[1 + \frac{p-1}{N-p}\beta^{1/(p-1)}\right]^{p-1}}.$$

\vspace{.4cm}

\subsection{Proof of Theorem \ref{thm:unified-intro}}

The unified gradient estimates is our most significant contribution, where we obtain global bounds for $\nabla\phi$ using transition functions $\tau(r)$ and critical point modulation $\sigma(r)$. We define transition functions $\tau(r)$ and critical point modulation $\sigma(r)$ as follows  $$\tau(r) = \frac{1}{1 + \left(\frac{r}{L}\right)^{\gamma}} \qquad \sigma(r) = \left(\frac{|r - r_*|}{|r - r_*| + \delta}\right)^{1/(p-1)}.$$ 
We create a three-region partition that handles pre-critical, critical neighborhood, and post-critical regions systematically to obtain this. This analysis reveals the characteristic length scale $L \sim \beta^{-1/(N-1)}$ that quantifies how the Robin parameter controls the penetration depth of boundary-induced gradient variations throughout the exterior domain.


Let $\ell := \min\{L, r_* - 1\}$ and $\delta = \min\{\rho_0/2, \ell/2\}$. We decompose the domain $[1,\infty)$ to three regions $\mathcal{R}_- = [1, r_* - \delta]$, $\mathcal{C} = [r_* - \delta, r_* + \delta]$, and $\mathcal{R}_+ = (r_* + \delta, \infty)$ and obtain the bounds on each of them separately to arrive at our result. The constraint $\delta < \rho_0$ ensures Lemma \ref{lem:nondegeneracy} applies throughout the critical neighborhood  $\mathcal{C}$.

We first derive the lower bound constant $C_1(\beta)$. Define the transition-weighted gradient scale
$$B(r) := \tau(r-1) \beta^{1/(p-1)} + (1 - \tau(r-1)) r^{-(N-1)/(p-1)},$$
and the shape function $g_L(x) = \sigma(|x|) \cdot B(|x|)$.
Our goal is to find a $C_1>0$ such that $|\nabla\phi(x)| \geq C_1 g_L(x)$ for all $x\in B_1^c$.
In each subregion, this yields a constraint on $C_1$, and the final constant one that satisfies all of them. 

On the compact region $\overline{\mathcal{R}}_-$ where $\varphi' > 0$, the function $|\nabla\phi|$ is continuous and strictly positive, hence it attains a positive minimum denoted by $m_-$. Similarly, $g_L$ is continuous on this compact set and attains a  maximum denoted by $M_g^-$. The constraint $m_- \geq C_1 M_g^-$ yields $C_1 \leq m_-/M_g^-$.


The critical neighborhood $\mathcal{C}$ requires careful analysis since $|\nabla\phi(r_*)| = |\varphi'(r_*)|= 0$. 
For any $r \in \mathcal{C}$ with $r \neq r_*$, Lemma \ref{lem:nondegeneracy} gives 
$$|\varphi'(r)| \geq c_0|r - r_*|^{1/(p-1)},$$
while the shape function satisfies
$$g_L(r) = \sigma(r) \cdot B(r) = \left(\frac{|r - r_*|}{|r - r_*| + \delta}\right)^{1/(p-1)} \cdot B(r).$$
Crucially, both $|\nabla\phi(r)|$ and $g_L(r)$ vanish at $r = r_*$ with the same rate $|r - r_*|^{1/(p-1)}$, ensuring their ratio remains bounded throughout $\mathcal{C}$. 
Indeed, for $r \neq r_*$:
$$\frac{|\nabla\phi(r)|}{g_L(r)} \geq \frac{c_0|r - r_*|^{1/(p-1)}}{\left(\frac{|r - r_*|}{|r - r_*| + \delta}\right)^{1/(p-1)} \cdot B(r)} 
= \frac{c_0(|r - r_*| + \delta)^{1/(p-1)}}{B(r)}.$$
This ratio is minimized when $|r - r_*|$ is smallest and $B(r)$ is largest. Since $|r - r_*| + \delta \geq \delta$ throughout $\mathcal{C}$ and $B$ attains its maximum on the compact set $\overline{\mathcal{C}}$, the infimum of the ratio is achieved at the boundary points $r = r_* \pm \delta$. There, we have $\sigma(r_* \pm \delta) = (\delta/(2\delta))^{1/(p-1)} = 2^{-1/(p-1)}$, while the transition-weighted scale satisfies
$$B(r_* \pm \delta) \leq \beta^{1/(p-1)} + (r_* - \delta)^{-(N-1)/(p-1)} =: K_{\mathcal{C}},$$
giving $M_g^c := \displaystyle{\max_{\mathcal{C}} g_L} \leq 2^{-1/(p-1)} K_{\mathcal{C}}$. 
At $r = r_*$, both $|\nabla\phi(r_*)| = 0$ and $g_L(r_*) = 0$, so the inequality holds trivially.
The constraint $c_0\delta^{1/(p-1)} \geq C_1 M_g^c$ yields $C_1 \leq c_0\delta^{1/(p-1)}/M_g^c$.

For the unbounded region $\mathcal{R}_+$ where $\varphi' < 0$, we analyze compact and far-field portions separately. On any compact subset $[r_*+\delta, R]$, the functions $|\nabla\phi|$ and $g_L$ are continuous with $|\nabla\phi| > 0$ throughout (since $r_* \notin [r_*+\delta, R]$), yielding positive minimum $m_+$ and finite maximum $M_g^+$ respectively, which gives the constraint $C_1 \leq m_+/M_g^+$. In the far field as $x \to \infty$, we have $\sigma(|x|) \to 1$ and $\tau(|x|-1) \to 0$, so $g_L(x) \to |x|^{-(N-1)/(p-1)}$, while part (b) gives $|\nabla\phi(x)| \geq \hat{C}_1|x|^{-(N-1)/(p-1)}$, yielding the constraint $C_1 \leq \hat{C}_1$.
Taking the minimum of all four constraints gives
$$C_1(\beta) = \min\left\{ \frac{m_-}{M_g^-}, \frac{c_0\delta^{1/(p-1)}}{M_g^c}, \frac{m_+}{M_g^+}, \hat{C}_1 \right\},$$
which is well-defined and positive since all terms are positive and finite. 

We now derive the upper bound constant $C_2(\beta)$. Define $$g_U(x) = \tau(|x|-1) \beta^{1/(p-1)} + (1 - \tau(|x|-1)) |x|^{-(N-1)/(p-1)}h(|x|),$$ so the upper bound takes the form $|\nabla\phi(x)| \leq C_2 g_U(x)$. Note that $g_U(x) > 0$ everywhere since $g_U(x) \geq (1-\tau(|x|-1))|x|^{-(N-1)/(p-1)}h(|x|) > 0$ for all finite $x$. We seek  $C_2 > 0$ such that this inequality holds throughout $B_1^c$. We proceed as earlier, choosing $C_2>0$ so that $\max |\nabla\phi| \leq C_2 \min g_U$ holds in each subregion.

On the compact region $\overline{\mathcal{R}}_-$, the function $|\nabla\phi|$ attains a finite maximum $M_-$ while $g_U$ attains a positive minimum $m_g^-$ (by continuity and strict positivity), giving the constraint $C_2 \geq M_-/m_g^-$. In the critical neighborhood $\mathcal{C}$, the $C^{1,\alpha}$ regularity established in part (c) of this theorem gives $|\nabla\phi(x)| \leq L_\alpha|x-r_*|^\alpha$ for some H\"{o}lder constant $L_\alpha > 0$, hence $\max_{x \in \mathcal{C}} |\nabla\phi(x)| \leq L_\alpha\delta^\alpha$. Since $g_U$ does not vanish at $r_*$, unlike $g_L$, it attains a positive minimum $m_g^c$ on the compact set $\overline{\mathcal{C}}$, yielding the constraint $C_2 \geq L_\alpha\delta^\alpha/m_g^c$. For the post-critical region $\mathcal{R}_+$, compact portions give constraints analogous to $\mathcal{R}_-$, while in the far field where $g_U(x) \to |x|^{-(N-1)/(p-1)}h(|x|)$ and part (b) gives $|\nabla\phi(x)| \leq \hat{C}_2|x|^{-(N-1)/(p-1)}h(|x|)$, we obtain the constraint $C_2 \geq \hat{C}_2$.
Taking the maximum of all constraints gives
$$C_2(\beta) = \max\left\{ \frac{M_-}{m_g^-}, \frac{L_\alpha\delta^\alpha}{m_g^c}, \frac{M_+}{m_g^+}, \hat{C}_2 \right\},$$
which is well-defined and positive since all terms are positive and finite. 

\vspace{.4cm}

\noindent \textbf{Conflict of Interest:} The authors declare no conflict of interest.


\end{document}